\documentclass[a4paper,reqno]{amsart}

\usepackage{color}
\usepackage{amsrefs}

\usepackage{amssymb}
\usepackage{latexsym}
\usepackage{amsmath}
\usepackage{amsthm}
\usepackage{graphicx}
\usepackage{tikz}
\usepackage{picture}

\def\sa{{\mathfrak a}}

   \def\cE{{\mathcal E}}   
      
      \def\cL{{\mathcal L}}

\def\cV{{\mathcal V}}

\def\cal H{{\mathcal H}}

\def\R{\mathbb{R}}
\def\C{\mathbb{C}}

\def\N{\mathbb{N}}
\def\phi{\varphi}
\def\d{\textup{d}}

\DeclareMathOperator{\Real}{Re}

\DeclareMathOperator{\spann}{span}

\DeclareMathOperator{\diam}{diam}

\renewcommand{\theta}{\vartheta}

\newtheorem{theorem}{Theorem}[section]
\newtheorem*{thm*}{Theorem}

\newtheorem{corollary}[theorem]{Corollary}
\newtheorem{lemma}[theorem]{Lemma}

\theoremstyle{definition}
\newtheorem{definition}[theorem]{Definition}
\newtheorem{example}[theorem]{Example}

\newtheorem{remark}[theorem]{Remark}

\numberwithin{equation}{section}

\title{Laplacians on bipartite metric graphs}

\author[P.~Kurasov]{Pavel Kurasov}

\author[J.~Rohleder]{Jonathan Rohleder}

\address{Matematiska institutionen\\ Stockholms universitet \\
106 91 Stockholm \\
Sweden}
\email{kurasov@math.su.se, jonathan.rohleder@math.su.se}

\begin{document}

\begin{abstract}
We study spectral properties of the standard (also called Kirchhoff) Laplacian and the anti-standard (or anti-Kirchhoff) Laplacian on a finite, compact metric graph. We show that the positive eigenvalues of these two operators coincide whenever the graph is bipartite; this leads to a precise relation between their eigenvalues enumerated with multiplicities and including the possible eigenvalue zero. Several spectral inequalities for, e.g., trees are among the consequences of this. In the second part we study inequalities between standard and Dirichlet eigenvalues in more detail and expose another connection to bipartiteness.
\end{abstract}

\maketitle

\section{Introduction}

Differential operators on metric graphs have attracted considerable attention in recent years. They have turned out to be useful as idealized models for systems on thin, network-like structures as, e.g., quantum wires or thin waveguides; see the recent monographs~\cite{BK13,M14} for more details and an overview on the vast literature. Any such operator is specified by the underlying metric graph, its action on functions on the graph and its vertex conditions. In this paper we consider a finite, compact metric graph $\Gamma$. As for the action of the differential operator on $\Gamma$, we focus on the Laplacian, i.e.\ the second derivative operator on each edge of the graph. The most common vertex conditions for the Laplacian on a metric graph are so-called standard (or continuity--Kirchhoff) conditions that require functions to be continuous at each vertex and to have balanced derivatives, i.e.\ the sum of all outgoing derivatives equals zero at each vertex; see Section~\ref{sec:dual} below for more details. The standard Laplacian $L^{\rm st} (\Gamma)$ provides an important example of a self-adjoint quantum graph.

One of our goals in the present paper is to investigate the relation between the spectra of the standard Laplacian and the anti-standard (or anti-Kirchhoff) Laplacian $L^{\rm a/st} (\Gamma)$ that corresponds to conditions that are formally dual to standard conditions: the derivative is assumed to be continuous and the values of the functions are balanced at each vertex. Some spectral properties of the anti-standard Laplacian were studied recently in~\cite{BL19,BM13,FuKuWi,ZS19}. However, it seemingly has remained unnoticed that there is an intimate and simple relation between the eigenvalues of $L^{\rm st} (\Gamma)$ and $L^{\rm a/st} (\Gamma)$. In fact, as we show in Section~\ref{sec:dual}, the eigenvalues of the two operators, enumerated non-decreasingly and counting multiplicities, satisfy
\begin{align}\label{eq:yeah!}
 \lambda_{k + \beta} \big(L^{\rm a/st} (\Gamma) \big) = \lambda_{k + 1} \big( L^{\rm st} (\Gamma) \big)
\end{align}
for all $k \in \N$ if (and only if) the graph $\Gamma$ is bipartite; here $\beta$ denotes the first Betti number, i.e.\ the number of independent cycles in $\Gamma$. To provide the simplest example, any tree is bipartite and satisfies $\beta = 0$, so that the identity~\eqref{eq:yeah!} holds and simplifies to $\lambda_k (L^{\rm a/st} (\Gamma)) = \lambda_{k + 1} (L^{\rm st} (\Gamma))$ for all $k$. The proof of the equation~\eqref{eq:yeah!} relies on a decomposition of the involved Laplacians into products of momentum operators.

The relation~\eqref{eq:yeah!} has a lot of immediate consequences for the spectra of the anti-standard Laplacian $L^{\rm a/st} (\Gamma)$ on any bipartite graph. However, it can also be used to derive properties of standard Laplacian eigenvalues. We demonstrate this by showing quite directly that on any tree $\Gamma$ one has
\begin{align}\label{eq:jaha}
 \lambda_{k + 1} \big(L^{\rm st} (\Gamma) \big) \leq \lambda_k \big(L^{\rm st, D} (\Gamma) \big)
\end{align}
for all $k \in \N$, where $L^{\rm st, D} (\Gamma)$ is the Laplacian subject to Dirichlet conditions at all vertices of degree one (i.e.\ on the ``natural boundary'' of $\Gamma$) and standard conditions at all other vertices; thus the two operators considered here differ only by their conditions on the boundary, being Neumann in one case and Dirichlet in the other. The inequality~\eqref{eq:jaha} was shown in~\cite{BBW15} by completely different methods; it is the counterpart for trees of an inequality between Neumann and Dirichlet eigenvalues of the Laplacian on a bounded domain in $\R^n$ due to Friedlander~\cite{F91} and Filonov~\cite{F04}. We point out that the inequality~\eqref{eq:jaha} is not true in general on graphs with $\beta \geq 1$.

This leads us to the second part of this paper: in Section~\ref{sec:Friedlander} we deal with variants of the inequality~\eqref{eq:jaha} for graphs that are not necessarily trees and compare the eigenvalues of the standard Laplacian with those of the ``decoupled'' Dirichlet Laplacian $L^{\rm D} (\Gamma)$, i.e.\ the Laplacian on $\Gamma$ subject to Dirichlet vertex conditions at all vertices. This gives rise to another connection to bipartiteness of the graph: in the case that $\Gamma$ is equilateral, i.e.\ all edges have the same length, we show that
\begin{align}\label{eq:soIstEs}
 \lambda_{k + 1} \big( L^{\rm st} (\Gamma) \big) \leq \lambda_k \big(L^{\rm D} (\Gamma) \big)
\end{align}
is valid for all $k \in \N$ if and only if $\Gamma$ is bipartite. This is done by using the relation between the eigenvalues of $L^{\rm st} (\Gamma)$ and those of a corresponding discrete Laplacian. For not necessarily equilateral graphs we provide sufficient conditions for the inequality~\eqref{eq:soIstEs} to hold for all $k$.

\section{Preliminaries}

Let $\Gamma$ be a finite, compact, connected metric graph formed by a finite set $\cE = \{ e_1, \dots, e_E \}$ of non-degenerate edges $e_n = [x_{2n-1}, x_{2n} ]$ 
joined at a
set of vertices $\cV = \{v_1, \dots, v_V \}$, where each vertex is understood as a subset of endpoints and the vertices are disjoint such that $ \{ x_j \}_{j=1}^{2 E} = v_1 \cup v_2 \cup \dots \cup v_M$. An edge is said to be incident to a vertex $v$ if (at least) one of its endpoints belongs to $v$. The degree $\deg v$ of a vertex $v$ in $\Gamma$ is the number of edges incident to $v$; note that loops, i.e.\ edges whose two endpoints belong to the same vertex, count twice. We introduce the natural boundary of $\Gamma$ being the set of vertices of degree one,
\begin{align*}
 \partial \Gamma := \left\{ v \in \cV : \deg v = 1 \right\} \subset \cV.
\end{align*}
The number of independent cycles (the first Betti number) is given by  $\beta := E - V + 1.$ In this paper we will be mainly concerned with graphs which are bipartite, i.e., $\cV = \cV_1 \cup \cV_2$ for disjoint sets $\cV_1, \cV_2$ and each edge has one endpoint in $\cV_1$ and one in $\cV_2$; see Figure~\ref{fig:bipartite}. 
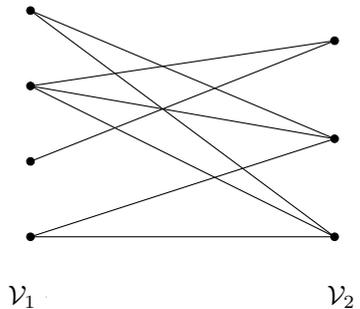
\begin{figure}[htb]

\setlength{\unitlength}{1cm}

\centering

\begin{tikzpicture}
\draw[fill] (2.2,-0.8) circle(0.00) node[left]{$\cV_1$};
\draw[fill] (2,0) circle(0.05);
\draw[fill] (2,1) circle(0.05);
\draw[fill] (2,2) circle(0.05);
\draw[fill] (2,3) circle(0.05);
\draw[fill] (6.4,-0.8) circle(0.00) node[left]{$\cV_2$};
\draw[fill] (6,0) circle(0.05);
\draw[fill] (6,1.3) circle(0.05);
\draw[fill] (6,2.6) circle(0.05);
\draw(2,3)--(6,1.3);
\draw(2,3)--(6,0);
\draw(2,2)--(6,0);
\draw(2,2)--(6,1.3);
\draw(2,2)--(6,2.6);
\draw(2,1)--(6,2.6);
\draw(2,0)--(6,0);
\draw(2,0)--(6,1.3);
\end{tikzpicture}
\caption{A bipartite graph.}
\label{fig:bipartite}

\end{figure}
Recall that a graph $\Gamma$ is bipartite if and only if each cycle in $\Gamma$ contains an even number of edges or, equivalently, if its chromatic number is two.

Our main goal is to study spectral properties of 
Laplacians on $ \Gamma $ in the case where the graph is bipartite. These operators will be self-adjoint in the Hilbert space
\begin{align*}
 L_2 (\Gamma) = \bigoplus_{n = 1}^E L_2 (e_n).
\end{align*}
They will be defined on subspaces of the Sobolev spaces $W^k_2 (\Gamma \setminus \cV)$, $k = 1, 2, \ldots$, consisting of functions in $L^2 (\Gamma)$ whose weak partial derivatives up to the order $k$ exist inside every edge and are square-integrable. The self-adjoint Laplacians under consideration will have purely discrete spectra and we will denote the eigenvalues of any such operator $A$ by
\begin{align*}
 \lambda_1 (A) \leq \lambda_2 (A) \leq \ldots,
\end{align*}
where we count multiplicities.

\section{Relation between standard and anti-standard eigenvalues and its consequences}\label{sec:dual}

\subsection*{Standard and anti-standard Laplacians}

With the Laplace differential expression
\begin{align} \label{diffexp}
 (\cL f) (x) = - f'' (x), \qquad  x \in e_n, \quad n = 1, \dots, E,
\end{align}
on $ \Gamma $ we associate two self-adjoint realisations. To specify their vertex conditions we denote by $ f(x_j) $ and $ \partial f (x_j) $ the limiting values of the function $ f $ and its first derivative (taken in the direction from the vertex into the edge) as the point $ x $ approaches one of the endpoints of the edge, i.e.\
\begin{align*}
 f(x_j)  & = \lim_{x \rightarrow x_j } f(x)
\end{align*}
and
\begin{align*}
  \partial f(x_j) & = \begin{cases}
 f'(x_j), & \text{provided $ x_j $ is the left endpoint},\\
 - f'(x_j), & \text{provided $ x_j $ is the right endpoint}.
\end{cases} 
\end{align*}

\begin{definition}
The {\em standard Laplacian} $ L^{\rm st} (\Gamma) $ (also called Kirchhoff or Neumann Laplacian in the literature) is defined by \eqref{diffexp} on the functions from the Sobolev space $ W_2^2 (\Gamma \setminus \cV) $ satisfying standard vertex conditions at each vertex $ v $, that is
\begin{align*}
\begin{cases}
 f(x_j) = f(x_i) \quad \text{provided} \; x_j, x_i \in v & \text{(continuity condition)}, \\
 \sum_{x_j \in v}  \partial f (x_j) = 0 & \text{(balance condition)}.
\end{cases}
\end{align*}
\end{definition}

\begin{definition}
The {\em anti-standard Laplacian} $ L^{\rm a/st} (\Gamma) $ (also called anti-Kirchhoff Laplacian) is defined by \eqref{diffexp} on
the functions from the Sobolev space $ W_2^2 (\Gamma \setminus \cV) $ satisfying anti-standard vertex
conditions, that is
\begin{align*}
 \begin{cases}
  \sum_{x_j \in v}  f (x_j) = 0 & \text{(balance condition)}, \\
  \partial f(x_j) = \partial f(x_i) \quad \text{provided} \; x_j, x_i \in v & \text{(continuity condition)}.
 \end{cases}
\end{align*}
\end{definition}

We remark that $L^{\rm st} (\Gamma)$ and $L^{\rm a/st} (\Gamma)$ are self-adjoint, non-negative operators in $L^2 (\Gamma)$ with purely discrete spectra. Furthermore, at any vertex of degree one the vertex conditions simplify to a Neumann condition in the case of $L^{\rm st} (\Gamma)$ or a Dirichlet condition for $L^{\rm a/st} (\Gamma)$. We further remark that the operators $L^{\rm st} (\Gamma)$ and $L^{\rm a/st} (\Gamma)$ both are independent of the choice of parametrization of the edges in $\Gamma$, of course assuming that the lengths are preserved.

\subsection*{Momentum operator decomposition of the Laplacians}

Let us introduce the momentum operator $ D $, which as a first order differential operator depends on the orientation of the edges. It was used earlier in \cite{FuKuWi} to derive index theorems for quantum graphs; cf.\ also~\cite{exner}. Our goal is to study bipartite graphs with the set of vertices $ \cV $ being divided into two disjoint sets $\cV_1 $ and $\cV_2 $ such that each edge connects a vertex in $\cV_1$ with a vertex in $\cV_2$. In what follows we shall assume that each edge is oriented pointing from $ \cV_1 $ to $ \cV_2$,
in other words the left endpoint of each edge belongs to a vertex from $ \cV_1 $ and the right one to a vertex from $ \cV_2. $

\begin{definition}
The {\em momentum operator} $ D = D (\Gamma) $ on $\Gamma$ is defined by
\begin{equation} \label{diffexp2} 
 (D f) (x) = \frac{1}{i} f'(x), \qquad x \in e_n, \quad n = 1, \dots, E,
\end{equation}
on the functions in the Sobolev space $ W_2^1 (\Gamma \setminus \cV) $ satisfying the continuity
condition
\begin{equation} \label{cont}
 f(x_j) = f(x_i) \quad \text{provided} \; x_j, x_i \in v
\end{equation}
at each vertex $v$.
\end{definition}

The momentum operator is uniquely determined by $ \Gamma $ since the orientation of
the edges is fixed; reversing the orientation of all edges simultaneously would lead to multiplication of $ D $ by $ -1 $. Note that the momentum operator is not self-adjoint in $L^2 (\Gamma)$ but its adjoint $ D^* $ is given by the same differential expression \eqref{diffexp2} on the functions from $ W_2^1 (\Gamma \setminus \cV) $ satisfying the balance
condition
\begin{equation} \label{balans}
\sum_{x_j \in v}  f(x_j) = 0
\end{equation}
 at each vertex $ v $.
 
The momentum operator can be used to express both the standard and anti-standard Laplacians; as a consequence, these operators are ``almost isospectral''.

\begin{lemma} \label{lem:isospec}
The standard Laplacian $ L^{\rm st} (\Gamma) $ and the anti-standard Laplacian \linebreak $ L^{\rm a/st} (\Gamma) $ on a bipartite, finite, compact metric graph $ \Gamma $ are related to the momentum operator $ D $ and its adjoint $ D^* $ via the relations
\begin{equation}\label{eq:DD}
L^{\rm st} (\Gamma)  = D^* D \qquad \text{and} \qquad L^{\rm a/st} (\Gamma) = D D^*.
\end{equation}
In particular, the positive eigenvalues of $ L^{\rm st} (\Gamma) $ and $ L^{\rm a/st} (\Gamma) $ coincide including multiplicities.
\end{lemma}

\begin{proof}
The representations~\eqref{eq:DD} are obvious from the definitions of the involved operators and the form of $D^*$ as $\Gamma$ is bipartite. Assume that $\lambda$ and $\psi$ are a positive eigenvalue and a corresponding eigenfunction, respectively, for the standard Laplacian, that is, $D^* D \psi = \lambda \psi$ by~\eqref{eq:DD}. Then $D \psi$ is non-trivial and belongs to the domain of $D D^*$ and, hence,
\begin{align*}
 L^{\rm a/st} (\Gamma) D \psi = D D^* D \psi = \lambda D \psi,
\end{align*}
that is, $D \psi$ is an eigenfunction of $L^{\rm a/st} (\Gamma)$ corresponding to the eigenvalue $\lambda$. Similarly, if $\phi$ is an eigenfunction of $L^{\rm a/st} (\Gamma)$ corresponding to a positive eigenvalue then $D^* \phi$ is an eigenfunction of $L^{\rm st} (\Gamma)$ corresponding to the same eigenvalue. As $D$ respectively $D^*$ map linearly independent eigenfunctions to linearly independent functions, multiplicities are preserved.
\end{proof}

\subsection*{Spectral relation between the Laplacians}

Despite Lemma~\ref{lem:isospec} the standard and anti-standard Laplacian are not necessarily isospectral in general. In order to establish the precise relation between the eigenvalues we need to study the kernels of these operators. The statement of the following lemma is also contained in~\cite[Lemma~2.1]{BL19}. However, for completeness of the presentation we provide a short proof.

\begin{lemma} \label{lem:Kernel}
The dimensions of the kernels of the standard and anti-standard Laplacians  
on a bipartite, finite, compact, connected metric graph $\Gamma$ are
\begin{equation}
\dim \ker L^{\rm st} (\Gamma) = 1 \qquad \text{and} \qquad \dim \ker L^{\rm a/st} (\Gamma) = \beta.
\end{equation}
\end{lemma}

\begin{proof}
Each function in the kernel of either $L^{\rm st} (\Gamma)$ or $L^{\rm a/st} (\Gamma)$ has to be constant on each edge. Indeed, if, e.g., $\psi \in \ker L^{\rm st} (\Gamma) = \ker D^* D$ then $0 = (D^* D \psi, \psi) = (D \psi, D \psi)$ in the inner product $( \cdot, \cdot)$ of $L^2 (\Gamma)$ and, hence, $D \psi = 0$; similarly for the anti-standard Laplacian. In particular, as $\Gamma$ is connected, the continuity condition at each vertex implies that each $\psi \in \ker L^{\rm st} (\Gamma)$ is constant on $\Gamma$ and, thus, $\dim \ker L^{\rm st} (\Gamma) = 1$.

Let $e_1, \dots, e_\beta$ be edges in $\Gamma$ such that removing $e_1, \dots, e_\beta$ from $\Gamma$ leads to a connected tree $T$. Then $e_j$ is part of a (unique up to shifts and inversion) cycle $C_j$ in $T \cup e_j.$  Since $ \Gamma $ is bipartite every cycle contains an even number of edges. Thus the function $\psi^j$ taking alternately the constant values $1$ and $- 1$ on the edges of $C_j$ and being constantly zero on the rest of $\Gamma$ belongs to the kernel of the anti-standard Laplacian since it satisfies the balance condition \eqref{balans}. In this way we obtain $ \beta $ linearly independent functions $ \psi^j $, all from the kernel. It remains to show that these functions span the kernel.

If $\psi \in \ker L^{\rm a/st} (\Gamma)$ is arbitrary then there exist constants $\gamma_1, \dots, \gamma_\beta \in \C$ such that $\gamma_j \psi^j$ coincides with $\psi$ on $e_j$, $j = 1, \dots, \beta$. Thus the function $\psi - \sum_{j = 1}^\beta \gamma_j \psi^j$ is supported on the tree $T$. It satisfies
Dirichlet conditions on each vertex of degree one in $T$ (including all vertices in $\partial \Gamma$). But every function which is constant on every edge of $T$ and satisfies the balance conditions \eqref{balans} is identically equal to zero on $T$. Thus $\psi^1, \dots, \psi^\beta$ form a basis of $\ker L^{\rm a/st} (\Gamma)$.
\end{proof}

Combining Lemma~\ref{lem:isospec} and Lemma~\ref{lem:Kernel} we arrive at the following result.

\begin{theorem}\label{thm:mainAst}
Assume that the finite, compact, connected metric graph $\Gamma$ is bipartite. Then
\begin{align*}
 \lambda_{k + \beta} \big( L^{\rm a/st} (\Gamma) \big) = \lambda_{k + 1} \big(L^{\rm st}(\Gamma) \big)
\end{align*}
holds for all $k \in \N$.
\end{theorem}

Observe that any graph can be transformed into a bipartite graph by introducing additional vertices inside certain edges. This procedure has no influence on the standard Laplacian but it changes the anti-standard Laplacian. Hence the statement of Theorem~\ref{thm:mainAst} cannot be extended to arbitrary metric graphs. In fact, if $\Gamma$ is not bipartite then $\dim \ker L^{\rm a/st} (\Gamma) = \beta - 1$ by~\cite[Lemma~2.1]{BL19} while $\dim \ker L^{\rm st} (\Gamma) = 1$ still holds and thus $\lambda_\beta (L^{\rm a/st} (\Gamma)) > 0 = \lambda_1 (L^{\rm st} (\Gamma))$.

The following statement is a consequence of Theorem~\ref{thm:mainAst} and observations on the kernel of the anti-standard Laplacian. For equilateral quantum graphs it was proved in~\cite[Corollary~3.9]{BM13}.

\begin{corollary}
Let the metric graph $\Gamma$ be finite, compact and connected. Then the following are equivalent:
\begin{enumerate}
 \item $ L^{\rm st} (\Gamma) $ and $ L^{\rm a/st} (\Gamma) $ are isospectral;
 \item $\Gamma$ is bipartite and $\beta = 1$.
\end{enumerate}
\end{corollary}

\begin{proof}
If $\Gamma$ is bipartite and $\beta = 1$ then it follows from Theorem~\ref{thm:mainAst} that the two operators are isospectral. 
Conversely, by Theorem~\ref{thm:mainAst} their isospectrality implies that $\beta = 1$ and $\dim \ker ( L^{\rm a/st}(\Gamma) ) = 1$. But then~\cite[Lemma~2.1]{BL19} gives that $\Gamma$ is bipartite.
\end{proof}

We would like to point out specifically the situation where $\Gamma$ is a finite metric tree as the spectrum of $L^{\rm a/st} (\Gamma)$ in this case was studied recently in~\cite{ZS19}.

\begin{corollary}\label{cor:tree}
Let $\Gamma$ be a finite, compact, connected metric tree. Then
\begin{align*}
 \lambda_{k} \big( L^{\rm a/st} (\Gamma) \big) = \lambda_{k + 1} \big(L^{\rm st}(\Gamma) \big)
\end{align*}
holds for all $k \in \N$.
\end{corollary}

As a consequence of this, all known results on the spectra of the standard Laplacian on a finite, compact metric tree carry over directly to the anti-standard Laplacian on the tree. To give a few examples, the lower eigenvalue bound 
\begin{align*}
 \lambda_k \big( L^{\rm a/st} (\Gamma) \big) \geq \frac{(k + 1)^2 \pi^2}{4 L (\Gamma)^2}, \qquad k \in \N,
\end{align*}
on any tree follows immediately from~\cite{F05}, see also~\cite{N87} and~\cite{KN14}; here $L (\Gamma)$ denotes the total length of $\Gamma$. Equality holds if and only if $\Gamma$ is an equilateral star with $k + 1$ edges. By means of Theorem~\ref{thm:mainAst} similar statements follow for any bipartite finite graph. More specifically for trees the upper estimates
\begin{align}\label{eq:tree1}
 \lambda_k \big( L^{\rm a/st} (\Gamma) \big) \leq \frac{k^2 \pi^2}{\diam (\Gamma)^2}, \qquad k \in \N,
\end{align}
and
\begin{align}\label{eq:tree2}
 \lambda_k \big( L^{\rm a/st} (\Gamma) \big) \leq \frac{k^2 E^2 \pi^2}{4 L (\Gamma)^2}, \qquad k \in \N,
\end{align}
provided $E \geq 2$, follow immediately from the corresponding results for the standard Laplacian in~\cite{R17}, where $\diam (\Gamma)$ is the diameter of $\Gamma$. In the latter estimate, equality holds for $k = 1$ if and only if $\Gamma$ is any equilateral star, and for $k > 1$ if and only if $\Gamma$ is an equilateral star with $E = 2$. The estimates~\eqref{eq:tree1}--\eqref{eq:tree2} were shown recently in~\cite{ZS19} in a more complicated way, mimicking the proofs for the standard Laplacian. We point out that upper estimates for anti-standard eigenvalues on general bipartite graphs can be derived from, e.g.,~\cite{BL17,BKKM17,BKKM19,K19,KKMM16} with the help of Theorem~\ref{thm:mainAst}.

Another example concerns estimates involving the doubly connected part of the graph -- the closed subgraph consisting of all $ x \in \Gamma $ for which there 
is a non-self-intersecting path in $ \Gamma $ starting and ending at $ x. $ Assume that the doubly connected part of $ \Gamma $ has size $ L_{\rm dc} \leq L(\Gamma) $. Then using \cite[Theorem 6.3]{BKKM19} we obtain the estimate
\begin{equation}
\lambda_{\beta +1} \big( L^{\rm a/st} (\Gamma) \big) \geq \lambda_2 \big( L^{\rm st} (\mathcal D) \big)
\end{equation}
for the eigenvalues of the anti-standard Laplacian, provided $ \Gamma $ is bipartite, where $ \mathcal D $ is the symmetric dumbbell graph of total length $ L(\Gamma) $ and both loops of length $ L_{\rm dc}/2. $ Similarly, if the doubly connected part of a bipartite graph $ \Gamma $ has a connected component of length $ L \leq L(\Gamma)$, then the estimate
\begin{equation}
\lambda_{\beta +1} \big( L^{\rm a/st} (\Gamma) \big) \geq \lambda_2 \big(L^{\rm st} (\mathfrak L) \big), 
\end{equation}
holds (following~\cite[Theorem 6.5]{BKKM19}), where $ \mathfrak L $ is the lasso graph of total length $ L(\Gamma) $ and length of the loop $L$.

It is a special feature of trees that functions satisfying standard vertex conditions at interior vertices can be transformed into functions with anti-standard interior vertex conditions by a simple transformation that does not increase the eigenvalues. More precisely, the following lemma holds; here $L^{\rm st, D} (\Gamma)$ denotes the Laplacian subject to standard vertex conditions at all vertices in $\cV \setminus \partial \Gamma$ and Dirichlet boundary conditions on $\partial \Gamma$.

\begin{lemma}\label{lem:innerTrafo}
If $\Gamma$ is a finite, compact, connected tree then 
\begin{align*}
 \lambda_k \big(L^{\rm a/st} (\Gamma) \big) \leq \lambda_k \big(L^{\rm st, D} (\Gamma) \big)
\end{align*}
holds for all $k \in \N$. 
\end{lemma}

\begin{proof}
Recall that the edges of $\Gamma$ are parametrized such that at each vertex either all incident edges are incoming or all are outgoing. Let us first show that we can choose numbers $\phi_1, \dots, \phi_E \in \R$ such that for each $v \in \cV \setminus \partial \Gamma$ we have
\begin{align}\label{eq:sumZero}
 \sum_{e_n~\text{incident to}~v} e^{i \phi_n} = 0.
\end{align}
Indeed, let $v_0 \in \partial \Gamma$ be arbitrary. Let $e_{n_0}$ be the edge incident to $v_0$ and let $v_1$ be the vertex different from $v_0$ to which $e_{n_0}$ is incident. Define $\phi_{n_0} = 0$ and assign the remaining unit roots of the equation $z^{\deg (v_1)} = 1$ to the edges $e_{n_1}, \dots, e_{n_{\deg (v) - 1}}$ incident to $v_1$ and different from $e_{n_0}$. Then the condition~\eqref{eq:sumZero} is satisfied for $v = v_1$. For the vertex $v_2$ different from $v_1$ to which $e_{n_1}$ is incident, assign to the edges incident to $v_2$ and different from $e_{n_1}$ the unit roots of $z^{\deg (v_2)} = 1$ different from $1$, multiplied by $e^{i \phi_{n_1}}$. Then the condition~\eqref{eq:sumZero} is satisfied also at $v = v_2$. Successively one can go through the whole tree and reach~\eqref{eq:sumZero} at each vertex $v \in \cV \setminus \partial \Gamma$.

Now define an operator $U : L^2 (\Gamma) \to L^2 (\Gamma)$ by
\begin{align*}
 (U f) (x) := e^{i \phi_n} f (x), \qquad x \in e_n, \quad n = 1, \dots, E.
\end{align*}
Clearly, $U$ is unitary. Consider the quadratic form
\begin{align*}
 \sa_U [f] & = \int_\Gamma |(U^* f)'|^2 \d x = \int_\Gamma |f'|^2 \d x
\end{align*}
defined on all functions $f \in W^1_2 (\Gamma \setminus \cV)$ such that $U^* f$ is continuous at each vertex and satisfies Dirichlet conditions on $\partial \Gamma$. This form is densely defined, nonnegative and closed with representing self-adjoint operator $U L^{\rm st, D} (\Gamma) U^*$. As the domain of this form is contained in the form domain of $L^{\rm a/st} (\Gamma)$ due to~\eqref{eq:sumZero} and the two forms have the same action on the smaller domain, the assertion of the lemma follows.
\end{proof}

The construction of the unitary operator $U$ in the previous proof does not work in general for bipartite graphs. Take, for instance, the lasso graph $\Gamma$ consisting of three edges $e_1 = [x_1, x_2], e_2 = [x_3, x_4], e_3 = [x_5, x_6]$ and three vertices $v_1 = \{x_1\}$, $v_2 = \{x_2, x_4, x_6\}$ and $v_3 = \{x_3, x_5\}$, see Figure~\ref{FigLasso}. Then $\Gamma$ is bipartite and in order to assign to each edge $e_n$ a complex number $z_n$ of modulus one such that the sum of these numbers is zero at each vertex one would need both $z_1 + z_2 + z_3 = 0$ and $z_2 + z_3 = 0$, a contradiction. 

\begin{figure}[htb]

\setlength{\unitlength}{1cm}

\centering

\begin{picture}(6,2)(0,0)
 \put(1.3,1){\circle{6}}
  \put(2,1){\line(1,0){3}}
 \put(1.5,0.7){$x_4$}
 \put(1.5,1.2){$x_6$}
 \put(2.1,0.6){$x_2$}
 \put(4.8,0.6){$x_1$}
 \put(0.2,0.7){$x_3$}
 \put(0.2,1.2){$x_5$}
\put(2,1){\circle*{0.1}}
\put(5,1){\circle*{0.1}}
\put(0.6,1){\circle*{0.1}}
\end{picture}

\caption{Lasso graph.}
\label{FigLasso}

\end{figure}
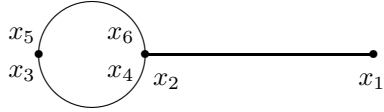

As a consequence of the previous lemma and Corollary~\ref{cor:tree} we get the following inequality between the eigenvalues of the standard Laplacian and the Laplacian $L^{\rm st, D} (\Gamma)$ subject to standard conditions on $\cV \setminus \partial \Gamma$ and Dirichlet conditions on $\partial \Gamma$.

\begin{corollary}\label{cor:Friedlander}
Let $\Gamma$ be a finite, compact, connected metric tree. Then 
\begin{align*}
 \lambda_{k + 1} \big(L^{\rm st} (\Gamma) \big) \leq \lambda_k \big(L^{\rm st, D} (\Gamma) \big)
\end{align*}
holds for all $k \in \N$.
\end{corollary}

This result was proven in~\cite[Lemma~4.5]{BBW15} by completely different methods (in the more general form given in Corollary~\ref{cor:FriedlanderB} below). It is the counterpart of the inequality between Neumann and Dirichlet Laplacian eigenvalues on domains proved by Friedlander in~\cite{F91} and refined by Filonov in~\cite{F04}.

\section{Generalisations}

The methods developed in the previous section for the standard and anti-standard Laplacians on bipartite graphs can be applied to larger classes of vertex conditions.

\subsection*{Mixed Dirichlet--Neumann conditions on the boundary}

The considerations in Section~\ref{sec:dual} were based on the fact that applying the momentum operator turns vertex conditions into their duals, i.e.\ standard conditions into anti-standard conditions and vice versa, provided the graph is bipartite. This principle extends naturally to mixed boundary conditions where some boundary vertices are equipped with Dirichlet conditions and the remaining ones with Neumann conditions. 

Let $\Gamma$ be a finite, compact, connected and bipartite metric graph, let $B \subset \partial \Gamma$ be any subset of boundary vertices and denote by $ L^{\rm st, D} (\Gamma, B) $ the Laplacian subject to Dirichlet conditions at the vertices in $B$ and standard conditions at all other vertices (including $ \partial \Gamma \setminus B$). Then $L^{\rm st, D} (\Gamma, B)$ allows a decomposition into a product of momentum operators, $L^{\rm st, D} (\Gamma, B) = D^* D$, where $D = \frac{1}{i} \frac{d}{dx}$ is defined on the functions in $W^1_2 (\Gamma \setminus \cV)$ that are continuous at all vertices and satisfy Dirichlet conditions on $B$. Then $D D^*$ coincides with the Laplacian $L^{\rm a/st,N} (\Gamma, B)$ subject to Neumann conditions on $B$ and anti-standard conditions on all remaining vertices (in particular, Dirichlet conditions on $\partial \Gamma \setminus B$). In particular, $L^{\rm st, D} (\Gamma, B)$ and $L^{\rm a/st, N} (\Gamma, B)$ have the same positive eigenvalues including multiplicities. If $B$ is nonempty then a reasoning similar to the proof of Lemma~\ref{lem:Kernel} yields
\begin{align*}
 \dim \ker L^{\rm st, D} (\Gamma, B) = 0 \quad \text{and} \quad \dim \ker L^{\rm a/st, N} (\Gamma, B) = \beta + |B| - 1.
\end{align*}
Hence we get the following counterpart of Theorem~\ref{thm:mainAst}.

\begin{theorem}
Let $\Gamma$ be a finite, compact, connected metric graph that is bipartite and let $B \subset \partial \Gamma$ be nonempty. Then 
\begin{align*}
 \lambda_{k + \beta + |B| - 1} \big(L^{\rm a/st, N} (\Gamma, B) \big) = \lambda_k \big(L^{\rm st, D} (\Gamma, B) \big)
\end{align*}
holds for all $k \in \N$.
\end{theorem}

As the reasoning of Lemma~\ref{lem:innerTrafo} also applies to mixed boundary conditions on $\partial \Gamma$, on any tree we get the following corollary analogous to Corollary~\ref{cor:Friedlander}, which is~\cite[Lemma~4.5]{BBW15} in its full generality, but with a different proof.

\begin{corollary}\label{cor:FriedlanderB}
Assume that $\Gamma$ is a finite, compact, connected metric tree. Then for $B \subset \partial \Gamma$
\begin{align}\label{eq:tree}
 \lambda_{k} \big(L^{\rm st,D} (\Gamma, B)\big) \leq \lambda_{k + |B| - 1} \big(L^{\rm st,D} (\Gamma, \partial \Gamma \setminus B ) \big)
\end{align}
holds for all $k \in \N$.
\end{corollary}

\subsection*{Scaling-invariant Laplacians} 

Standard and anti-standard conditions are special cases of scaling-invariant (or non-Robin) vertex conditions. At a vertex of degree $d$ such conditions are characterized by two mutually orthogonal subspaces $ X^\pm (v) $ that span $ \mathbb C^{d} $, and the vertex conditions are then given by
\begin{align} 
\label{scinvvc}
\begin{cases}
 \vec{f} (v) \in X^+ (v), & \\
 \partial \vec{f}(v)  \in X^-(v), &
\end{cases}
\end{align}
where $\vec{f} (v) $ is the vector containing the values $f (x_j)$ for all $x_j \in v$ and $\partial \vec{f} (v) $ is the vector containing all the $\partial f(x_j)$, where the endpoints $x_j$ are enumerated in the same order.
For $X^+ (v) = \spann \{(1, \dots, 1)^\top\}$ and $X^- (v) = X^+ (v)^\perp$ we have standard conditions, and the converse choice leads to anti-standard conditions. Any scaling-invariant Laplacian $L$ is then defined on all functions in $ W_2^2 (\Gamma \setminus V) $ satisfying the conditions~\eqref{scinvvc}.
Under these conditions, stretching all the edges in $ \Gamma $ and keeping the same vertex conditions
will lead to a Laplace operator whose eigenvalues are multiples of the eigenvalues for the original problem, hence the name scaling-invariant.

In the same way as for the operators in Section~\ref{sec:dual} every scaling-invariant Laplacian $L$ possesses a decomposition into a product of two mutually adjoint first order operators, 
\begin{equation} \label{prod}
 L = D^* D ,
 \end{equation}
where $ D = \frac{1}{i} \frac{d}{dx} $ is defined on all functions $f$ in $W^1_2 (\Gamma \setminus \cV)$ such that $\vec f (v) \in X^+ (v)$ is satisfied at each vertex. If the graph $ \Gamma $ is bipartite, then the domain of the adjoint operator  $ D^* $ is determined by the conditions
$$ \vec{f} (v) \in X^-(v) $$
at each vertex. Hence the domain of the dual operator $ \hat{L} := D D^* $ is given by just interchanging the roles of the subspaces $ X^+(v) $ and $ X^-(v) $ at each vertex. Thus a similar analysis as in Section~\ref{sec:dual} can be carried out. The only difficulty is to determine the dimensions of the kernels of $L$ and $\hat L$.

If $ \Gamma $ is not bipartite, then still every scaling-invariant Laplacian possesses a decomposition \eqref{prod}, but the domain of the dual operator $ \hat{L} $ depends on the orientation of the edges and is not obtained by just interchanging the subspaces $X^+ (v)$ and $X^- (v)$.

\subsection*{Topological perturbations of bipartite graphs}

Standard and anti-standard conditions lead to a different behaviour of the eigenvalues as one of the vertices is chopped into two pieces, i.e.\ dividing the corresponding equivalence class into two or more and thus turning $\Gamma$ into a new graph $\Gamma'$: 
\begin{itemize}
 \item the eigenvalues of $L^{\rm st} (\Gamma)$ are non-increasing, $\lambda_k (L^{\rm st} (\Gamma)) \geq \lambda_k (L^{\rm st} (\Gamma'))$ for all $k \in \N$, since the domain of the quadratic form is increasing;
 \item the eigenvalues of $L^{\rm a/st} (\Gamma)$ are non-decreasing, $\lambda_k (L^{\rm a/st} (\Gamma)) \leq \lambda_k (L^{\rm a/st} (\Gamma'))$ for all $k \in \N$, since  the domain of the quadratic form is decreasing,
\end{itemize} 
see, e.g.,~\cite{KMN13,RS19}. On the other hand, as we have shown $L^{\rm st} (\Gamma)$ and $L^{\rm a/st} (\Gamma)$ have the same spectra outside the origin, and the same holds for $L^{\rm st} (\Gamma')$ and $L^{\rm a/st} (\Gamma')$.

The relation between the eigenvalues on the chopped graph $\Gamma'$ depends on whether the cut leads to 
a graph with two connected components or just opens one of the cycles.  In either case we have the relation
$$ \lambda_{m+1} \big( L^{\rm st} (\Gamma') \big) = \lambda_{m + \beta - 1} \big( L^{\rm a/st} (\Gamma') \big) , \qquad m \in \N,$$
but for different reasons:
\begin{itemize}
 \item if $ \Gamma' $ is connected, then the multiplicity of the eigenvalue zero for the anti-standard Laplacian decreases by $ 1 $ since $ \beta' = \beta -1 ;$
 \item if $ \Gamma' $ is not connected, then the multiplicity of the eigenvalue zero for the standard Laplacian increase by $1$ since $ \Gamma' $ consists of two components.
\end{itemize}

\section{Inequalities between standard and Dirichlet eigenvalues}\label{sec:Friedlander}

This section is devoted to comparison principles of the type of Corollary~\ref{cor:Friedlander}. It will turn out that among equilateral quantum graphs bipartite graphs are characterized by the validity of such an inequality. Moreover, we consider non-equilateral cases. In contrast to the situation in Corollary~\ref{cor:Friedlander}, in this section we focus on the comparison between the eigenvalues of the standard Laplacian $L^{\rm st} (\Gamma) $ and the {\em Dirichlet Laplacian} $L^{\rm D} (\Gamma)$ that acts as $L^{\rm D} (\Gamma) f = \cL f$ and is defined on all functions $f \in  W^2_2 (\Gamma\setminus \cV)$ such that 
\begin{align*}
 f (v) = 0 \quad \text{for all}~v \in \cV.
\end{align*}
These vertex conditions separate the graph effectively into intervals and its spectrum consists of the points
\begin{align*}
 \frac{m^2 \pi^2}{L (e_n)^2}, \quad m \in \N, n = 1, \ldots, E,
\end{align*}
where $ L(e_n) $ is the length of the edge $ e_n. $ That is, the spectrum is determined by the edge lengths and is independent of the connectivity of $\Gamma$. It is clear from variational principles that we have the two trivial inequalities 
\begin{align*}
 \lambda_n \big(L^{\rm st} (\Gamma)\big) \leq \lambda_n \big(L^{\rm D} (\Gamma) \big)
\end{align*}
and
\begin{align}\label{eq:ganzTrivial}
 \lambda_n \big(L^{\rm a/st} (\Gamma)\big) \leq \lambda_n \big(L^{\rm D} (\Gamma) \big)
\end{align}
for all $n \in \N$. Using the second estimate and Corollary~\ref{cor:tree}, for any tree we obtain
\begin{align}\label{eq:forTrees}
 \lambda_{n + 1} (L^{\rm st}(\Gamma)) \leq \lambda_n \big(L^{\rm D} (\Gamma)  \big)
\end{align}
for all $n \in \N$, which is another simple proof of~\cite[Theorem~4.1]{R17}. Note that for bipartite graphs with $ \beta > 0 $ the combination of the second inequality with Theorem~\ref{thm:mainAst} does not improve the estimate for the standard Laplacian. The estimate~\eqref{eq:forTrees} is not true in general on a metric graph $\Gamma$. To provide one of the simplest counterexamples, consider the loop graph constisting of one edge $[x_1, x_2]$ and one vertex $v = \{x_1, x_2\}$. In this case we have $\lambda_2 (L^{\rm st} (\Gamma)) = \frac{4 \pi^2}{(x_2 - x_1)^2}$ but $\lambda_1 (L^{\rm D} (\Gamma)) = \frac{\pi^2}{(x_2 - x_1)^2}$, that is,~\eqref{eq:forTrees} is violated for $n = 1$.  In this context we also refer the reader to the recent observation~\cite{F19}.

We would like to point out that the estimate~\eqref{eq:forTrees} cannot be improved, since for the equilateral star graph $S$ with three edges and edge lengths $1$ we have
\begin{align*}
 & \lambda_n \big( L^{\rm st} (S) \big) = 0, \Big( \frac{\pi}{2} \Big)^2,  \Big( \frac{\pi}{2} \Big)^2,   \pi ^2, \Big( \frac{3 \pi}{2} \Big)^2, \dots,  \\
 & \lambda_n \big( L^{\rm D} (S) \big) = \pi^2, \pi^2, \pi^2, (2 \pi)^2, \dots,
\end{align*}
implying $\lambda_{4} (L^{\rm st}(S)) = \lambda_3 (L^{\rm D} (S) )$ and $\lambda_5 (L^{\rm st} (S)) > \lambda_3 (L^{\rm D} (S))$. The eigenvalues depend continuously on the edge lengths, hence considering any three-star graph with almost equal but rationally independent edge lengths
we obtain a counterexample telling that estimate \eqref{eq:forTrees} cannot be improved by imposing any extra conditions like rational independence of the edge lengths. For the same reason, the equivalent inequality~\eqref{eq:ganzTrivial} cannot be improved either by imposing such conditions. In particular, the estimate claimed in~\cite[Theorem~5]{ZS19} seems to be false.

\subsection*{Equilateral graphs} For equilateral graphs the following theorem gives a complete answer to the question for which $n$ the inequality~\eqref{eq:forTrees} is satisfied. It leads back to the notion of bipartiteness.

\begin{theorem}\label{thm:malSehen}
Let the finite, compact, connected metric graph $\Gamma$ be equilateral. Then the inequality
\begin{align}\label{eq:Friedlander}
 \lambda_{n + 1} \big( L^{\rm st}(\Gamma) \big) \leq \lambda_n \big(L^{\rm D} (\Gamma) \big)
\end{align}
holds for all $n \in \N$ if and only if $\Gamma$ is bipartite. Moreover, if $\Gamma$ is not bipartite then \eqref{eq:Friedlander} is violated for
all $ n = (2m+1) E, \; m =0,1,2, \dots$, but holds for all other $ n. $ 
\end{theorem}

\begin{proof}
Assume without loss of generality that each edge has length one. Then the spectrum of $L^{\rm D} (\Gamma)$ is
given by the eigenvalues $ (\pi m )^2$, $m = 1,2, \dots$, where each eigenvalue has multiplicity $E$. Hence the first $2 E$ eigenvalues of $L^{\rm D} (\Gamma)$ are given by
\begin{align} \label{LD}
\begin{split}
 & \lambda_1 \big(L^{\rm D} (\Gamma) \big) = \lambda_2 \big( L^{\rm D} (\Gamma) \big) = \dots = \lambda_E \big(L^{\rm D} (\Gamma) \big) = \pi^2; \\
 & \lambda_{E + 1} \big(L^{\rm D} (\Gamma) \big) = \lambda_{E + 2} \big(L^{\rm D} (\Gamma) \big) = \dots = \lambda_{2 E} \big(L^{\rm D} (\Gamma) \big) = 4 \pi^2.
\end{split}
\end{align}

To determine the spectrum of the standard Laplacian $L^{\rm st} (\Gamma)$ one may use the formula connecting it to the spectrum of the
normalised  Laplacian $ L_{\rm norm} (G)$ on the discrete graph $G$ having the same set of vertices and edges as $\Gamma$, defined as
\begin{align*}
 \big( L_{\rm norm} (G) \psi \big) (v) = \psi (v) - \frac{1}{\sqrt{{\rm deg}\; v}} \sum_{w \sim v}   \frac{1}{\sqrt{{\rm deg}\; w}} \psi (w),
\end{align*}
where $v \sim w$ means that the vertex $w$ is adjacent to $v$.\footnote{In the case of parallel edges and loops the summaton should be taken over all edges connecting vertices. In this way parallel edges are counted in accordance with their multiplicity.  Loops contribute twice to both summation and calculation of the vertex degree.}  The spectrum of $L_{\rm norm} (G)$ consists of $V$ eigenvalues $0 = \mu_1 < \mu_2 \leq \dots \leq \mu_V \leq 2$. They are related to the eigenvalues of $L^{\rm st} (\Gamma)$ corresponding to eigenfunctions that do not vanish at all vertices simultaneously: a number $k_j^2 > 0$ is an eigenvalue of $L^{\rm st} (\Gamma)$ with an eigenfunction that is not identically zero on all vertices if and only if 
\begin{align}\label{eq:discCont}
 1 - \cos k_j = \mu_n, \qquad n = 1, 2, \dots, V,
\end{align}
see~\cite{vB85}. To each eigenvalue of the normalised Laplacian on a discrete graph correspond infinitely many eigenvalues of the standard Laplacian on the metric graph, but inside the interval $ (0, \pi^2) $ the correspondence to the eigenvalues of $L_{\rm norm} (G)$ in $(0, 2)$ is one-to-one including multiplicities. Note that eigenfunctions of $L^{\rm st} (\Gamma)$ that vanish simultaneously on all vertices may only appear at eigenvalues of the form $(n \pi)^2$ with $n \in \N$.

We will also make use of the fact that $\mu = 2$ is an eigenvalue of $L_{\rm norm} (G)$ if and only if the graph $G$ is bipartite, see, e.g.,~\cite[Lemma 1.7]{Chung}. For bipartite graphs the  multiplicity of the eigenvalue $2$ is one and the corresponding eigenfunction is equal to $1$ on $\cV_1$ and $- 1$ on $\cV_2$ if $\cV_1$, $\cV_2$ form a partition of $\cV$ such that each edge connects a vertex in $\cV_1$ with a vertex in $\cV_2$. 

Let us discuss the eigenvalues of the standard Laplacian on the interval $ [0, 4 \pi^2]. $
We are going to show that there are totally $ 2 E + 1 $ eigenvalues (which can also be seen from the  Weyl asymptotics). It will be convenient to separately discuss the eigenvalues inside the open intervals $ (0, \pi^2) $ and $ (\pi^2, 4 \pi^2) $ and  the points $ 0, \pi^2, 4 \pi^2 .$

{\bf Case 1.} Consider the eigenvalues of $L^{\rm st} (\Gamma)$ inside $(0, \pi^2) \cup (\pi^2, 4 \pi^2)$. To every eigenvalue $\mu$ of $L_{\rm norm} (G)$ that lies in $(0, 2)$ there correspond precisely two numbers $k \in (0, 2 \pi)$ satisfying~\eqref{eq:discCont} situated symmetrically with respect to the middle point $\pi$ of $(0, 2 \pi)$, and its squares $k^2$ are eigenvalues of $L^{\rm st} (\Gamma)$. Moreover, among the $V$ eigenvalues of $L_{\rm norm} (G)$, $0$ is an eigenvalue of multiplicity one and, according to the above remark, $2$ is either an eigenvalue of multiplicity one (if $G$ is bipartite) or no eigenvalue. Accordingly, $L^{\rm st} (\Gamma)$ has
\begin{itemize}
\item $V - 1$ pairs of eigenvalues in $(0, \pi^2) \cup (\pi^2, 4 \pi^2)$ if $G$ is non-bipartite, 
\item $V - 2$ pairs of eigenvalues in $(0, \pi^2) \cup (\pi^2, 4 \pi^2)$ if $G$ is bipartite;
\end{itemize}
for each pair, one of the values belongs to $(0, \pi^2)$ and one to $(\pi^2, 4 \pi^2)$.

{\bf Case 2.}  The point $ k^2 = 0 $ is always an eigenvalue of multiplicity one since the graph is connected.

{\bf Case 3.} The point $ k^2 = \pi^2 $ is an eigenvalue of multiplicity $ \beta \pm1 $ depending on whether $ G $ is bipartite or not. To determine the multiplicity we need to calculate the number of eigenfunctions corresponding to $ k^2 = \pi^2. $ If the graph is bipartite, then there is one eigenfunction not equal to zero at the vertices and given by $  \cos \pi (x-x_{2n-1}) $ on every edge $ e_n $, provided the edges are oriented so that their left end points belong to the same bipartite component of $G$. On the other hand, if the graph is not bipartite, then $ \mu = 2 $ is not an eigenvalue of $ L_{\rm norm} (G)$ and only eigenfunctions that vanish at all vertices may exist.

It remains to calculate the number of eigenfunctions which are equal to zero at all vertices. Every cycle with an even number of edges determines one such eigenfunction equal
to $ \pm \sin \pi (x-x_{2n-1}) $ on every edge $ e_n $ in the cycle and zero outside. If $ G $ is bipartite then all cycles have an even number of edges and we obtain $ \beta $
linearly independent eigenfunctions. There are no other such eigenfunction (see \cite{KuJFA}).

If the graph $ G $ is not bipartite then there exists at least one cycle formed by an odd number of edges. There is no eigenfunction supported only on such a cycle and
equal to zero at the vertices, but any two cycles with an odd number of edges determine precisely one eigenfunction equal to zero at the vertices: it is supported by the two cycles
and any single path connecting them. Therefore we have $ \beta -1 $ eigenfunctions equal to zero at the vertices in this case.

Summing up, the multiplicity of $ \lambda = \pi^2$ equals
\begin{itemize}
\item $ \beta - 1 = E - V$ if $G$ is non-bipartite;
\item $ \beta + 1 = E - V + 2$ if $G$ is bipartite.
\end{itemize}

{\bf Case 4.}  The point $k^2 = 4 \pi^2 $ is an eigenvalue of multiplicity $1 + \beta = E - V + 2$. Indeed, there is just one eigenfunction equal to $ 1 $ at all vertices (given by $ \cos 2 \pi (x-x_{2n-1})$ on every edge $e_n$) and $ \beta $ eigenfunctions  equal to zero at all vertices; each such eigenfunction is supported on one of the independent cycles in $\Gamma$ (given by $\sin 2 \pi (x-x_{2n-1})$ on each edge $ e_n $ in the cycle).

To sum up, there are precisely $2 E + 1$ eigenvalues of the standard Laplacian inside the interval $ [0, 4 \pi^2] $, and they satisfy the following:
\begin{itemize}
 \item  if the graph  \underline{is not} bipartite, then
\begin{align} \label{Lst1}
\begin{split}
 & 0 = \lambda_1 \big( L^{\rm st} (\Gamma) \big) < \dots \leq \lambda_V \big( L^{\rm st} (\Gamma) \big) < \lambda_{V + 1} \big( L^{\rm st} (\Gamma) \big) = \dots = \lambda_{E} \big(L^{\rm st} (\Gamma) \big) = \pi^2; \\
 & \pi < \lambda_{E + 1} \big(L^{\rm st} (\Gamma) \big) \leq \dots < \lambda_{E + V} \big( L^{\rm st} (\Gamma) \big) = \dots = \lambda_{2 E + 1} \big( L^{\rm st} (\Gamma) \big) = 4 \pi^2.
\end{split}
\end{align}
\item  if the graph \underline{is} bipartite, then
\begin{align} \label{Lst2}
\begin{split}
 & 0 = \lambda_1 \big( L^{\rm st} (\Gamma) \big) < \dots \leq \lambda_{V - 1} \big( L^{\rm st} (\Gamma) \big) < \lambda_{V} \big( L^{\rm st} (\Gamma) \big) = \dots = \lambda_{E + 1} \big( L^{\rm st} (\Gamma) \big) = \pi^2; \\
 & \pi < \lambda_{E + 2} \big( L^{\rm st} (\Gamma) \big) \leq \dots < \lambda_{E + V} \big( L^{\rm st} (\Gamma) \big) = \dots = \lambda_{2 E + 1} \big( L^{\rm st} (\Gamma) \big) = 4 \pi^2.
\end{split}
\end{align}
\end{itemize}
A comparison between \eqref{Lst1} and \eqref{LD} implies that the inequality \eqref{eq:Friedlander} holds for all $n \leq 2 E$ except for $n = E$  if $G$ is non-bipartite. From comparing \eqref{Lst2} with \eqref{LD} we conclude that the inequality \eqref{eq:Friedlander} holds for all $n \leq 2 E$  if $G$ is bipartite. 

In order to cover higher eigenvalues we observe that in $ k $-scale the spectrum in any interval $ (2\pi m, 2\pi (m+1)] $ is obtained by shifting the interval $ (0, 2 \pi ] $ to the right  due to the $2 \pi$-periodicity of the relation~\eqref{eq:discCont} for generic points $ k \neq \pi n$, $n \in \mathbb N$, and repeating our analysis for the special points $k = \pi n$, $n \in \mathbb N$. 
Hence the inequality \eqref{eq:Friedlander} holds for any $n$, provided the graph
is bipartite, and is violated exactly for $n = (2 m + 1) E$, provided the graph is non-bipartite.
\end{proof}

\subsection*{Non-equilateral graphs}

For more general, possibly non-equilateral graphs the validity of the inequality~\eqref{eq:Friedlander} is not related one-to-one to bipartiteness of $\Gamma$. Consider for instance the cycle graph consisting of two edges with lengths $1$ and $3$. Then $\Gamma$ is bipartite but $\lambda_2 (L^{\rm st} (\Gamma)) = \frac{\pi^2}{4} > \frac{\pi^2}{9} = \lambda_1 (L^{\rm D} (\Gamma))$.

However, one may give sufficient conditions in terms of the relation of the edge lengths within each cycle for the eigenvalue inequality~\eqref{eq:forTrees} to hold. This is done in the following theorem. There we say that $\Gamma$ contains only independent cycles if there is no edge in $\Gamma$ which is part of two different cycles. Moreover, for each cycle $C$ we write $\cE (C)$ for the set of edges which form the cycle. The length of any edge $e \in \cE$ is denoted here by $L (e)$.

\begin{theorem}\label{thm:gluing}
Assume that the finite, compact, connected metric graph $\Gamma$ contains only independent cycles. Moreover, for each cycle $C$ in $\Gamma$ assume that for each $\hat e \in \cE (C)$ there exist numbers $\nu_{\hat e} (e) \in \{-1, 1\}$, $e \in \cE (C)$, such that 
\begin{align}\label{eq:important}
 \frac{1}{L (\hat e)} \sum_{e \in C} \nu_{\hat e} (e) L (e) \in 2 \N.
\end{align}
Then 
\begin{align}\label{eq:Friedlander'}
 \lambda_{n + 1} \big(L^{\rm st}(\Gamma) \big) \leq \lambda_n \big(L^{\rm D} (\Gamma) \big)
\end{align}
holds for all $n \in \N$.
\end{theorem}

\begin{proof}
Let us first consider the case that $\Gamma$ consists only of one cycle, that is, all vertices in $\Gamma$ have degree two. Let $n \in \N$ be arbitrary. Then $\lambda := \lambda_n (L^{\rm D} (\Gamma)) = m^2 \pi^2 / L (\hat e)^2$ for some $\hat e \in \cE$ and $m \in \N$. By the min-max principle for $L^{\rm D} (\Gamma)$ we have
\begin{align*}
 \lambda = \min_{\substack{F \subset W^1_{2, 0} (\Gamma \setminus \cV) \\ \dim F = n}} \max_{\substack{f \in F \\ f \neq 0}} \frac{\int_\Gamma |f'|^2 d x}{\int_\Gamma |f|^2 d x},
\end{align*}
where $W^1_{2, 0} (\Gamma \setminus \cV)$ consists of all $f \in W^1_2 (\Gamma \setminus \cV)$ such that $f (v) = 0$ for all $v \in \cV$. Hence there exists an $n$-dimensional subspace $F$ of $W^1_{2, 0} (\Gamma \setminus \cV)$ such that
\begin{align}\label{eq:formMu}
 \int_\Gamma |f'|^2 d x \leq \lambda \int_\Gamma |f|^2 d x, \qquad f \in F.
\end{align}
Assume that a parametrization of the edges and an enumeration $e_1, \dots, e_{E}$ of $\cE$ is chosen along the orientation and order of the cycle. Define a function $g$ on $\Gamma$ by
\begin{align*}
 g (x) = e^{i \sqrt{\lambda} \left(\sum_{k = 1}^{j - 1} \nu_{\hat e} (e_k) L (e_k) + \nu_{\hat e} (e_j) (x - x_{2 j - 1}) \right)}, \qquad x \in e_j, \quad j = 1, \dots, E.
\end{align*}
Then $g \in W^2_2 (\Gamma \setminus \cV)$ and $g$ is continuous at each vertex since for $j = 1, \dots, E - 1$ we have
\begin{align*}
 g (x_{2 j}) = e^{i \sqrt{\lambda} \left(\sum_{k = 1}^j \nu_{\hat e} (e_k) L (e_k) \right)} = g (x_{2 j + 1})
\end{align*}
and 
\begin{align*}
 g (x_{2 E}) & = e^{i \sqrt{\lambda} \left(\sum_{k = 1}^{E} \nu_{\hat e} (e_k) L (e_k) \right)} = e^{i m \pi \frac{1}{L (\hat e)} \left(\sum_{k = 1}^{E} \nu_{\hat e} (e_k) L (e_k) \right)} = 1 = g (x_1)
\end{align*}
by \eqref{eq:important}. Note that $g$ satisfies
\begin{align}\label{eq:wichtig}
 g' (x) = i \sqrt{\lambda} \, \nu_{\hat e} (e_j) g (x) \quad \text{and} \quad - g'' (x) = \lambda g (x), \qquad x \in e_j, j = 1, \dots, E.
\end{align}

Let now $f \in F$ and $\eta \in \C$. Then the inequality~\eqref{eq:formMu}, integration by parts and~\eqref{eq:wichtig} yield
\begin{align}\label{eq:jetztGehtsLos}
\begin{split}
 \int_\Gamma |f' + \eta g'|^2 d x & = \int_\Gamma |f'|^2 d x + 2 \Real \int_\Gamma f' \overline{\eta g'} d x + \int_\Gamma |\eta g'|^2 d x \\
 & \leq \lambda \int_\Gamma |f|^2 d x + 2 \Real \int_\Gamma f \overline{\eta (- g'')} d x + \int_\Gamma |\eta g'|^2 d x \\
 & = \lambda \int_\Gamma |f|^2 d x + 2 \lambda \Real \int_\Gamma f \overline{\eta g} d x + \lambda \int_\Gamma |\eta g|^2 d x \\
 & = \lambda \int_\Gamma |f + \eta g|^2 d x,
\end{split}
\end{align}
where we have used that $f$ vanishes at each vertex. Note that $|g (x)| = 1$ for all $x \in \Gamma$ and, hence, $g \notin F$. Thus~\eqref{eq:jetztGehtsLos} implies the assertion~\eqref{eq:Friedlander} in the case that $\Gamma$ is a cycle.

Let now $\Gamma$ be an arbitrary compact, finite, connected graph having only independent cycles. Then $\Gamma$ can be obtained 
by gluing together successively cycle graphs having the property~\eqref{eq:important} and trees, where in each step the gluing may only take place
 at one fixed vertex. Using the statement for cycle graphs and the inequality~\eqref{eq:forTrees} for trees it suffices to show the following: If $\Gamma_1$, $\Gamma_2$ are any finite, compact graphs such that $\lambda_{n + 1} (L^{\rm st} (\Gamma_j)) \leq \lambda_n ( L^{\rm D}(\Gamma_j))$ holds for all $n \in \N$, $j = 1, 2$, then 
\begin{align}\label{eq:goal}
 \lambda_{n + 1} \big( L^{\rm st} (\Gamma_{1,2}) \big) \leq \lambda_n \big( L^{\rm D} (\Gamma_{1,2}) \big) \quad \text{for all}~n \in \N,
\end{align}
where $\Gamma_{1,2}$ is any graph obtained from choosing one vertex of $\Gamma_1$ and one vertex of $\Gamma_2$ and gluing together $\Gamma_1$ and $\Gamma_2$ at these vertices. The inequality~\eqref{eq:goal} follows from a perturbation argument. Indeed, fix $n \in \N$. Since the Dirichlet Laplacian on $\Gamma_{1,2}$ is the direct sum of the Dirichlet Laplacians on $\Gamma_1$ and $\Gamma_2$, there exist numbers $m, j \in \N$ such that $m + j \geq n$ and
\begin{align*}
 & \lambda_m \big(L^{\rm D}(\Gamma_1) \big) \leq \lambda_n \big(L^{\rm D} (\Gamma_{1,2}) \big) < \lambda_{m + 1} \big(L^{\rm D}(\Gamma_1) \big), \\
 & \lambda_j \big( L^{\rm D}(\Gamma_2) \big) \leq \lambda_n \big(L^{\rm D}(\Gamma_{1,2}) \big) < \lambda_{j + 1} \big(L^{\rm D}(\Gamma_2) \big).
\end{align*}

Then it follows from the assumption on the eigenvalues on $\Gamma_1$ and $\Gamma_2$ that $L^{\rm st}(\Gamma_1)$ has at least $m + 1$ eigenvalues 
in the interval $[0, \lambda_n (L^{\rm D}(\Gamma_{1,2}))]$ and $L^{\rm st} (\Gamma_2)$ has at least $j + 1$ eigenvalues in the interval 
$[0, \lambda_n (L^{\rm D}(\Gamma_{1,2}))]$. As the standard Laplacian on $\Gamma_{1,2}$ is a rank one perturbation in the resolvent sense of the direct 
sum of $L^{\rm st} (\Gamma_1)$ and $L^{\rm st} (\Gamma_2)$, it follows that $L^{\rm st} (\Gamma_{1,2})$ has at least $m + j + 1 \geq n + 1$ 
eigenvalues in $[0, \lambda_n (L^{\rm D} (\Gamma_{1,2}))]$. From this the inequality~\eqref{eq:goal} follows.
\end{proof}

\begin{remark}
If $\Gamma = C$ is a cycle graph and contains a pair of rationally independent edge lengths then the condition~\eqref{eq:important} implies $\sum_{e \in E} \nu_{\hat e} (e) L (e) = 0$ for each $\hat e$.
\end{remark}

The following example shows that the condition~\eqref{eq:important} is not necessary for the inequality~\eqref{eq:Friedlander'} to hold for all $k$.

\begin{example}
Consider the graph consisting of a cycle formed by four edges $e_1, \dots, e_4$ with lengths $L (e_1) = 5, L (e_2) = 3$ and $L (e_3) = L (e_4) = 2$. Then the condition~\eqref{eq:important} is not satisfied. Indeed, there is no possibility to choose numbers $\nu (e_j) \in \{-1, 1\}$, $j = 1, 2, 3, 4$, such that
\begin{align*}
 \frac{1}{5} \big(5 \nu (e_1) + 3 \nu (e_2) + 2 \nu (e_3) + 2 \nu (e_4) \big) \in 2 \N.
\end{align*}
On the other hand, one can check that the inequality~\eqref{eq:Friedlander'} is satisfied for all $n \in \N$.
\end{example}

Observe that in the condition~\eqref{eq:important} the numbers $\nu_{\hat e} (e)$ have to be chosen for each edge $\hat e$ in each cycle. The following corollary gives an easier sufficient condition for~\eqref{eq:important}.

\begin{corollary}
Let $\Gamma$ be a graph containing only independent cycles. Moreover, for each cycle $C$ of $\Gamma$ assume that there exist numbers $\nu (e) \in \{-1, 1\}$, $e \in \cE (C)$, with
\begin{align*}
 \sum_{e \in C} \nu (e) L (e) = 0.
\end{align*}
Then 
\begin{align*}
 \lambda_{n + 1} \big(L^{\rm st}(\Gamma) \big) \leq \lambda_n \big(L^{\rm D} (\Gamma) \big)
\end{align*}
holds for all $n \in \N$.
\end{corollary}

We give another example for the case of rationally dependent edge lengths.

\begin{example}
Let $\Gamma$ be a cycle graph with pairwise rationally dependent edge lengths. Then there exists $x \in \R$ such that $x L (e) \in \N$ holds for each edge $e$. We assume that $x$ is minimal with this property, i.e., $\gcd \{x L (e) : e \in \cE\} = 1$. Suppose $x L (G) = \sum_{e \in \cE} x L (e)$ is odd. Without loss of generality assume that $L (e)$, $e \in \cE$, are natural numbers with $\gcd \{L (e) : e \in \cE\} = 1$. Consider the cycle $\widetilde \Gamma$ obtained from $\Gamma$ by dividing each edge $e$ into $L (e)$ edges of length one. Then $\widetilde \Gamma$ is an equilateral cycle graph with an odd total number $\widetilde E$ of edges, and with the help of Theorem~\ref{thm:malSehen} it follows
\begin{align*}
 \lambda_{\widetilde E + 1} \big(L^{\rm st} (\Gamma) \big) = \lambda_{\widetilde E + 1} \big(L^{\rm st} (\widetilde \Gamma) \big) > \lambda_{\widetilde E} \big(L^{\rm D} (
 \widetilde \Gamma) \big) \geq \lambda_{\widetilde E} \big(L^{\rm D} (\Gamma) \big).
\end{align*}
Thus the inequality~\eqref{eq:Friedlander'} is violated for $n = \widetilde E$.
\end{example}

\section*{Acknowledgements} 
The work of PK was supported by the Swedish Research Council (VR) grant D0497301.
JR gratefully acknowledges financial support by the grant No.\ 2018-04560 of the Swedish Research Council (VR).



\end{document}